\documentclass[a4paper,10pt]{amsart} 
\usepackage[english]{babel}
\usepackage{dsfont} 
\usepackage{graphicx}
\usepackage{color}
\usepackage{enumerate}
\usepackage{array,multirow}

\allowdisplaybreaks

\usepackage{amssymb} 
\usepackage{amsthm} 
\usepackage{dsfont}

\newtheorem{thm}{Theorem}
\newtheorem{dfn}[thm]{Definition}
\newtheorem{lem}[thm]{Lemma}

\newtheorem{exm}[thm]{Example}

\newtheorem{prop}[thm]{Proposition}
\newtheorem{rem}[thm]{Remark}
\newtheorem{cor}[thm]{Corollary}

\newcommand{\dN}{\mathds{N}}

\newcommand{\dR}{\mathds{R}}

\newcommand{\alg}{\mathsf{alg}}
\newcommand{\cone}{\mathsf{cone}}
\newcommand{\ch}{\mathsf{char}}

\author{Konrad Schm\"udgen}
\address{University of Leipzig, Mathematical Institute, Augustusplatz 10/11, D-04109 Leipzig, Germany}
\email{\tt schmuedgen@math.uni-leipzig.de}

\date{}
\begin{document}

\begin{title}{On the Archimedean Positivstellensatz in Real Algebraic Geometry}
\end{title}
\date{\today}

\begin{abstract} A variant of the Archimedean Positivstellensatz is proved which is based on Archimedean semirings or quadratic modules of generating subalgebras. It allows one to obtain representations of strictly positive polynomials on compact semi-algebraic sets by means of smaller sets of squares or polynomials. A large number of examples is developed in detail.
\end{abstract}
\maketitle


\textbf{Key  words:} Positivstellensatz, positive polynomial, semi-algebraic set, quadratic module, semiring, moment problem

\section{Introduction} 
Archimedean Positivstellens\"atze are important tools of real algebraic geometry to describe  positive polynomials on compact semi-algebraic sets in terms of weighted sums of squares or as elements of semirings (see e.g \cite[Section 7.4]{marshall} or \cite[Section 12.4]{sch17}). In many situations there are finer versions of these theorems for which smaller sets of squares or elements are sufficient.
Some results in this direction have been obtained in \cite[Sections 4 and 6]{kss}. The aim of this paper is to provide a variant of a general Archimedean Positivstellensatz of this kind and to treat a number of illustrating examples.

To formulate the result we  need some terminology and facts from real algebra  \cite{marshall}, \cite{PD}. Throughout this paper,
$A$ is  a {\bf commutative unital real algebra} and $A_1,\dots,A_m$, $m\in \dN$, are subalgebras of $A$ which contain the unit element $1$ of $A$. For $\alpha\in \dR$, we  write $\alpha$ for  $\alpha\cdot 1$. The subalgebra of $A$ generated by $A_1,\dots,A_n$, $n\leq m,$, is denoted by $\alg(A_1,\dots,A_n)$.

By  a {\it cone} in $A$ we mean a subset $C$ of $A$ such that $\lambda c\in C$ and $c+d\in C$ for all $c,d\in C$ and $\lambda\geq 0$. A cone $C$ in $A$ is called {\it unital} if $1\in C$.

If $C_j$ is a unital cone of  $A_j$ for $j=1,\dots,n$, then $\cone(C_1,\dots,C_n)$ denotes the unital cone in  $\alg(A_1,\dots,A_n)$ which is generated by $C_1\cdot C_2\cdots C_n$. That is,  $\cone(C_1,\dots,C_n)$ is the set of elements
\begin{align*}
\sum_{k=1}^r\,  c_{1k}\cdots c_{nk},~~\textrm{ where}~~ c_{jk}\in C_j,\, j=1,\dots,n,\, k=1,\dots,r,\, r\in \dN.
\end{align*}
We recall the following standard notions of real algebraic geometry.
\begin{dfn}
  A unital cone $C$ in $A$ is called a\\
  $\bullet$\, \emph{quadratic module} of\, $A$ if\, $a^2b \in C$  for all $b\in C$ and $a\in A.$\\
  $\bullet$\, \emph{semiring} (or a \emph{preprime}) of\, $A$ if\, $C\cdot C\subseteq C$.\\
  $\bullet$\, \emph{preordering} of $A$ if\, $C$ is a quadratic module and a semiring.
\end{dfn}

 The bounded part of $A$ with respect to a unital cone $C$   is defined by
\begin{align*}
A_{\rm bd}(C):=\{ a\in A:~ \textrm{there exists}~\lambda> 0~~\textrm{such that}~ ~(\lambda-a) \in C~~ \textrm{and} ~~ ( \lambda+ a)\in C\}.
\end{align*}
A unital cone $C$ of  $A$  is called {\it Archimedean} if $A=A_{\rm bd}(C)$. If $C$ is a quadratic module or a semiring of $A$, then $A_{\rm bd}(C)$ is a subalgebra of $A$. Thus, in this case   $C$ is Archimedean if a set of generators of the algebra $A$ belongs to $A_{\rm bd}(C)$.

For a unital cone $C$  of $A$, we define 
\begin{equation*}
    C^\dagger := \{ a\in A:~a+\epsilon \in C ~ \textup{ for all }~  \epsilon \in (0,\infty)\, \} .
\end{equation*}
Obviously, $C^\dagger$ is again a unital cone of $A$ and $C\subseteq C^\dagger$. Such cones were introduced in \cite{kumarshall} and studied extensively in \cite{kss}.

A {\it character} of $A$ is a unital algebra homomorphism of $A$ into $\dR$. The set of all characters of $A$ is denoted by $\ch(A)$. For a subset $C$ of $A$, let
\begin{align*}
\ch(A;C) :=\{ \chi \in \ch(A):\, \chi(c)\geq 0~ \textrm{for all}~ c\in C\, \}.
\end{align*}

A unital cone $M$ of $A$ is called a {\it $C$-module} for a semiring $C$ of $A$ if $ac\in M$ for all $a\in M$ and $c\in C$. Since  $1\in M$, we then have $C\subseteq M$. \smallskip

The following  theorem is the main result of this paper.
\begin{thm}\label{positivst}
Suppose that the  algebra $A$ is generated by the subalgebras $A_1,\dots,A_m$ of $A$. We assume  that   $C_i$ is an Archimedean semiring or  an Archimedean quadratic module of the algebra $A_i$ for $i=1,\dots,m$. Let $M$ be a unital cone of $A$. Then for any element $a\in A$ the following are equivalent:
\begin{itemize}
\item[(i) ]~ $\chi(a)>0$ for   all  $\chi\in \ch(A;M)$ and  $\chi\in \ch(A;C_i)$,  $i=1,\dots,m$.
\item[(ii)] ~ There exists a number $\varepsilon>0$ such that $a\in \varepsilon +\cone(M,C_1,\dots,C_m).$
\end{itemize}
\end{thm}

We restate the result in the important special case $M=\dR_+\cdot 1$  separately as
\begin{cor}\label{casem}
Retain the assumptions of Theorem \ref{positivst} and let $a\in A$. The the following two statements are equivalent:
\begin{itemize}
\item[(i) ]~ $\chi(a)>0$ for   all    $\chi\in \ch(A;C_i)$ and  $i=1,\dots,m$.
\item[(ii)] ~ There is a number $\varepsilon>0$ such that $a\in \varepsilon +\cone(C_1,\dots,C_m).$
\end{itemize}
\end{cor}
\begin{proof}
Set $M=\dR_+\cdot 1 $. Then  $\cone(M,C_1,\dots,C_m)=\cone(C_1,\dots,C_m)$ and   all characters of $A$ are in $\ch(A;M)$,  so the assertion follows  from 
Theorem \ref{positivst}.
\end{proof}
The following corollary is the main assertion of Theorem \ref{positivst} in the case $m=1$. It  can be considered as a Positivstellensatz for "modules" of quadratic modules.
\begin{cor}\label{casem1}
Suppose $Q$ is an Archimedean quadratic module of $A$ and let $M$ be a unital cone of $A$. If $a\in A$ satisfies $\chi(a)>0$ for   all  $\chi\in \ch(A;M)$ and  $\chi\in \ch(A;Q)$,  then $a\in \varepsilon +\cone(M,Q)$ for some $ \varepsilon>0$.
\end{cor}
Let us briefly discuss  the preceding results and notions. 
If all cones $C_i$ are semirings of $A_i$, then it is easily seen that $\cone(C_1,\dots,C_m)$ is semiring of $A=\alg(A_1,\dots,A_m)$. However, if some cones $C_i$ are quadratic modules and others are semirings which are not preorderings, then  $\cone(C_1,\dots,C_m)$ is in general neither a semiring nor a quadratic module of $A$ (see Example \ref{n3nosemiring} below). In fact, if $C_i$ is a quadratic module of $A_i$, then $C_i$ is invariant under multiplication by squares of $A_i$, but not under multiplication by squares of the larger algebra $A$. However, if all $C_i$ are Archimedean semirings or quadrartic modules, then $\cone(C_1,\dots,C_m)^\dagger)$ is even a preordering. This is shown in the next theorem which is also a main ingredient of the proof of Theorem \ref{positivst}.
\begin{thm}\label{daggerpre}
Under the assumptions of Theorem \ref{positivst}, 
$\cone(C_1,\dots,C_m)^\dagger$ is an\newline Archimedean preordering of the algeba $A$ and $\cone(M,\cone(C_1,\dots,C_m)^\dagger)$ is an $\cone(C_1,\dots,C_m)^\dagger$-module.
\end{thm}
 
The proofs of both theorems are given in the next section. A large number of examples, known and new ones, are developed in section \ref{examples}.

\begin{rem}
The results of this paper have direct applications to the classical multi-dimensional moment problem. For this we recall Haviland's theorem \cite{haviland}, in the general form  given in \cite[Theorem 1.14]{sch17}:

Suppose $A$ is a finitely generated unital real algebra and $K$ is a closed subset of $\ch(A)$. Let $C$ denote the preordering 
\begin{align}\label{preord}
C:=\{a \in A: \chi(a)> 0, \chi\in K\}.
\end{align} If $L$ is a linear functional on $A$ such that  $L(a)\geq  0$ for all $a\in C$, then $L$ is a $K$-moment functional, that is, there exists a Radon measure $\mu$ on $\ch(A)$ supported on $K$ such that all functions of $ A$ are $\mu$-integrable and $L(f)=\int f d\mu$ for $f\in A$.

Based on Theorem \ref{positivst} and Corollaries \ref{casem} and \ref{casem1} we develop in Section \ref{examples} descriptions of the preordering $C$ given by (\ref{preord})  in terms of defining polynomials of the corresponding semi-algebraic sets $K(f_1,\dots,f_r)$. Requiring that $L(a)\geq 0$ for all these elements $a$ yield new solvability criteria for the $K(f_1,\dots,f_r)$-moment problem in the corresponding examples. We leave it to the reader to restate these criteria.
\end{rem}

\section{Proofs of Theorems \ref{positivst} and  \ref{daggerpre}}

First we recall two  results which will be essentially used in the proofs given below. The following is the {\it Positivstellensatz for modules of Archimedean semirings}. The version for Archimedean semirings (the case $C=S$) is due to Krivine \cite{krivine}.
\begin{prop}\label{possemiring}
Suppose that $C$ is an $S$-module of an Archimedean semiring $S$ of a commutative unital real algebra $A$. Let $a\in A$. If $\chi(a)>0$ for all $\chi\in \ch(A;C)$, then $a\in C$.
\end{prop}
\begin{proof}
\cite{jacobi} or
\cite[Theorem 5.4.4]{marshall} or \cite[Theorem 12.43]{sch23}.
\end{proof}

\begin{prop}\label{preordlemma}
If $C$ is an $S$-module of an Archimedean semiring $S$ or $C$ is an Archimedean quadratic module of $A$, then $C^\dagger$ is a preordering.
\end{prop}
\begin{proof}
\cite[Theorem 3.2 and Corollary 3.6]{kss}.
\end{proof}
The proofs of  the theorems  are divided into several steps stated as lemmas. In the rest of this section we retain the assumptions of Theorem \ref{positivst} and the notation introduced above.

\begin{lem}\label{archi}
If each cone $C_i$ is Archimedean in $A_i$ for $=1,\dots,n$, then  $\cone(C_1,\dots,C_n)$  is an Archimedean cone in $\alg(A_1,\dots, A_n)$.
\end{lem}
\begin{proof} Let us abbreviate $D_k:=\cone(C_1,\dots,D_k)$ and $B_k:=\alg(A_1,\dots,A_k)$. We prove by induction that $D_k$ is an Archimedean  cone in $B_k$ for $k=1,\dots,n$. For $k=1$ this is clear, since $D_1=C_1$ is Archimedean in $A_1$  by assumption. Assume that  the assertion holds for $k<n$. Let $a\in B_{k+1}$. Since $B_{k+1}=\alg(A_{k+1},B_k)$,  $a$ is of the form 
$\sum_{i=1}^n a_i b_i$, where $a_\in A_{k+1}$ and $B_i\in B_k$ for $i=1,\dots,n$. Since $C_{k+1}$ is Archimedean in $A_{k+1}$ and $D_k$ is Archimedean in $B_k$ by induction assumption, there exists a $\lambda >0$ such that  $\lambda\pm a_i\in C_{k+1}$ and $\lambda\pm b_i\in D_k$ for $i=1,\dots,n$. Then
\begin{align}
3n\lambda^2 -a= \sum_{i=1}^n (\lambda- a_i)(\lambda +b_i) + \lambda \sum_{i=1}^n (\lambda+ a_i) +\lambda\sum_{i=1}^n (\lambda -b_i) \in D_{k+1}
\end{align}
which completes the induction proof.+
\end{proof}
\begin{lem}\label{product1} If $c,d\in \cone(C_1,\dots,C_m)$ and $\varepsilon>0$, then $cd+\varepsilon \in \cone(C_1,\dots,C_m)$.
\end{lem}
\begin{proof}
We retain the notation from the proof of Lemma \ref{archi} and prove the assertion by induction. That is, we show that $cd+\varepsilon \in D_k$ for $c,d\in D_k$ and $\varepsilon>0$. First let $k=1$. Then, since $C_1$ is an $S_1$-module of an Archimedean semiring $S_1$ or $C_1$ is an Archimedean quadratic module in $A_1$, $C_1^\dagger$ is a preordering in $A_1$ by 
Proposition \ref{preordlemma}, so 
$c,d\in D_1=C_1$ implies $cd\in C_1^\dagger$ and hence $cd+\varepsilon \in C_1=D_1$. Suppose now that the assertion is true for $n<m$. Let $c,d\in D_{n+1}$. Then $c$ and $d$ can be written as
\begin{align*}
c=\sum_{k=1}^r c_k u_k ~~\textrm{and}~~ d=\sum_{k=1}^r d_k v_k, ~~~\textrm{where} ~~ c_k, d_k\in C_{n+1}, u,v_k\in D_n.
\end{align*}
Since $C_{n+1}$ is Archimedean in $A_1$ by assumption and $D_n$ is Archimedean in the algebra $B_n=\alg(A_1,\dots,A_n)$ by Lemma \ref{archi}, there exists a number $\lambda >0$ 
such that $(\lambda-c_j d_k)\in C_{n+1}$ and $(\lambda-u_jv_k)\in D_n$ for all $j,k=1,\dots,r$. Further, because $C_{n+1}^\dagger$ is a preordering, again by Proposition \ref{preordlemma}, we have $c_jd_k\in C_{n+1}^\dagger$ and hence $c_jd_j+\delta \in C_{n+1}$ for any $\delta>0$. By the induction hypothesis, $u_jv_k+\delta\in D_n$ for  $\delta>0$ and $j,k=1,\dots,r$. Now we set
 $\delta:=-\lambda +\sqrt{\varepsilon r^{-2}+\lambda^2}$. Then  $\delta>0$ and $\delta^2r^2+2\delta \lambda r^2=\varepsilon$. Using the latter equation we obtain
\begin{align*}
cd +\varepsilon = \sum_{j,k=1}^r (c_jd_k+\delta)(u_jv_k+\delta)+ \delta\sum_{j,k=1}^r (\lambda- c_jd_k) +\delta\sum_{j,k=1}^r (\lambda- u_jv_k).
\end{align*}
By construction, the first sum is an element of $\cone(C_{n+1},D_n)=D_{n+1}$, the second is in $C_{n+1}$, and the third is in $D_n$. Therefore, 
$cd+\varepsilon \in D_{n+1}$, which completes the induction proof. The case $n+1=m$ gives the assertion.
\end{proof}
\begin{lem}\label{product2}
If $c,d  \in \cone(C_1,\dots,C_m)^\dagger$ then $cd\in \cone(C_1,\dots,C_m)^\dagger$.
\end{lem}
\begin{proof}
We mimick the reasoning of the proof of the preceding lemma. By Lemma \ref{archi},  $\cone(C_1,\dots,C_m)$ is Archimedean in $A$. Hence there exists $\lambda >0$ 
such that $(\lambda -c)\in \cone(C_1,\dots,C_m)$ and  $(\lambda -d)\in \cone(C_1,\dots,C_m)$. Let $\varepsilon >0$ and put $\delta:=-\lambda +\sqrt{\varepsilon/2+\lambda^2}$. 
Then  $\varepsilon=\delta^2+2\delta \lambda +\varepsilon/2$. Since $c,d  \in \cone(C_1,\dots,C_m)^\dagger$, we have $(c+\delta), (d+\delta)\in \cone(C_1,\dots,C_m)$. Therefore,  it follows from   Lemma \ref{product1} that
 $ [(c+\delta) (d+\delta)+\frac{1}{2}\varepsilon/2 ]\in \cone(C_1,\dots,C_m)$. Using the preceding facts and
\begin{align}
cd+\varepsilon=[(c+\delta)(d+\delta)+\varepsilon/2]+ \delta (\lambda- c) +\delta (\lambda- d).
\end{align}
we conclude that $cd+\varepsilon\in \cone(C_1,\dots,C_m)$ for all $\varepsilon>0$. This means that $cd\in \cone(C_1,\dots,C_m)$.
\end{proof}

\begin{lem}\label{squares} For each $a\in A$, we have $a^2\in \cone(A_1,\dots,A_m)^\dagger$.
\end{lem}
\begin{proof}
Since $A$ is generated the subalgebras  $A_1,\dots,A_m$ , $a\in A$ is of the form 
\begin{align}\label{aform}
a=\sum_{k=1}^n a_{1k}\cdots a_{mk}\, ,~~\textrm{ where}~~ a_{ik}\in A_i,\, i=1,\dots,m, k=1,\dots, n. n\in \dN.
\end{align}
Without loss of generality we assume that $C_i$ is an $S_i$-module of a semiring $S_i$ for $i=1,\dots,r$ and $C_{r+1},\dots,C_m$ are quadratic modules.
 Since each  $C_l$  is an Archimedean cone of $A_l$ by assumption,  there exists a positive number $\lambda$ such that
\begin{align}\label{ajbjsq}
(\lambda-a_{jk}), (\lambda+a_{jk})\in C_j~~\textrm{and} ~~~(\lambda-a_{i1}^2-\cdots- a_{in}^2)\in C_i
\end{align} 
for $j=1,\dots,r$, $i=r+1,\dots,m$, and $k=1,\dots,n$.  

Now we consider the polynomial algebra $B:=\dR[x_{11},x_{12},\dots,x_{mn}]$ and the unital cone $S$ in $B$ obtained by the non-negative combination of polynomials
\begin{align*}
(\lambda-x_{11})^{k_{11}}(&\lambda+x_{11})^{l_{11}}\cdots (\lambda-x_{1n})^{k_{1n}}(\lambda+x_{1n})^{l_{1n}}\times \cdots \times\\ (\lambda-x_{r1})^{ k_{r1}}&(\lambda+x_{r1})^{l_{r1}}\cdots (\lambda-x_{rn})^{k_{rn}}(\lambda+x_{rn})^{l_{rn}}\times \\ 
(\lambda-x_{r+1,1}^2-\cdots-x_{r+1,n}^2) &q_{r+1}(x_{r+1,1},\dots,x_{r+1,n}) + p_{r+1}(x_{r+1,1},\dots,x_{r+1,n})\times \cdots\times\\
(\lambda-x_{m1}^2-\cdots&-x_{mn}^2) q_{m}(x_{m1},\dots,x_{mn}) + p_{m}(x_{m1},\dots,x_{mn}),
\end{align*}
where $k_{11},\dots,k_{rn},l_{11},\dots,l_{rn}\in \dN_0$ and 
\begin{align}\label{ssos}
q_j,p_j\in \sum \dR[x_{j1},\dots,x_{jn}]^2~~\textrm{ for}~~  j=r+1,\dots,m.
\end{align}
Since the product of such polynomials is  of the same form, $S$ is a semiring in $B$. We  verify that $AS$ is Archimedean in $B$. By definition,  $\lambda \pm x_{ji}$ for $j=1,\dots,r,$  $i=1,\dots,n$ belongs to $S$. Now we turn to $x_{ji}$, with $j\in \{r+1,\dots,m\}, i\in \{1,\dots,n\}$. Since  $ \lambda-x_{j1}^2-\cdots-x_{jn}^2$  is in $S$ by definition, it follows from the identity
\begin{align*}
\sqrt{\lambda}\pm x_{ji}=\frac{1}{2\sqrt{\lambda}}\Big[(\lambda- x_{j1}^2-\cdots- x_{jn}^2) +\sum_{l=1,l\neq i}^n x_{jl}^2 +(\sqrt{\lambda}\pm x_{ji})^2\Big]
\end{align*}
that $\sqrt{\lambda}\pm x_{ji}$ is in $S$. Therefore,  $S$ is an Archimedean semiring of $B=\dR[x_{11},\dots,x_{mn}]$.

Let $\varepsilon >0$ and consider the polynomial 
$$p_\varepsilon(x_{11},\dots,x_{mn}):=\Big(\sum_{k=1}^n x_{1k}\cdots x_{mk}\Big)^2+ \varepsilon$$ 
of $\dR[x_{11},\dots,x_{mn}]$. Obviously, $p_\varepsilon>0$ on  $\dR^{mn}$. In particular, $p_\varepsilon>0$ on the semi-algebraic set of $\dR^{mn}$ defined by the semiring $S$. 
Therefore, by the Archimedean positivstellensatz for semirings (Proposition \ref{possemiring} for $C=S$), $ p_\varepsilon$ belongs to $S$.

Setting $x_{ji}=a_{ji}$  in $ p_\varepsilon(x_{11},\dots,x_{mn})$ and using  (\ref{aform}) we conclude  that  $a^2+\varepsilon$ is a non-negative combination of elements
\begin{align*}
(\lambda-a_{11})^{k_{11}}(&\lambda+a_{11})^{l_{11}}\cdots (\lambda-a_{1n})^{k_{1n}}(\lambda+a_{1n})^{l_{1n}}\times \cdots \times\\ (\lambda-a_{r1})^{ k_{r1}}&(\lambda+a_{r1})^{l_{r1}}\cdots (\lambda-a_{rn})^{k_{rn}}(\lambda+a_{rn})^{l_{rn}}\times \\
(\lambda-a_{r+1,1}^2-\cdots-a_{r+1,n}^2) &q_{r+1}(a_{r+1,1},\dots,a_{r+1,n}) + p_{r+1}(a_{r+1,1},\dots,a_{r+1,n})\times \cdots\times\\ 
(\lambda-a_{m1}^2-\cdots &-a_{mn}^2) q_{m}(a_{m1},\dots,a_{mn}) + p_m(a_{m1},\dots,a_{mn}),
\end{align*}
 For $i=1,\dots,r$ we have $( \lambda \pm a_{ij})\in C_i$ by (\ref{ajbjsq}) and $C_i$ is a semiring by assumption. Hence, the product of all terms in first two lines  belongs to $C_1\cdots C_r$.
From (\ref{ssos}) it follows that  $ q_j(a_{j,1},\dots,a_{jn})$ and $ p_j(a_{j,1},\dots,a_{jn})$  are in $\sum (A_j)^2.$ for $j=r+1,\dots,m$. Since $C_j$  is a quadratic module of $A_j$ for $j=r+1,\dots,m$  we conclude that  the product of terms in the last two lines is in $C_{r+1}\cdots C_m$. Hence the whole product  belongs to $\cone(C_1,\dots,C_m)$. Therefore, $a^2+\varepsilon\in \cone(C_1,\dots,C_m)$ for all $\varepsilon>0$. That means that $a^2 \in \cone(C_1,\dots,S_m)^\dagger$.
\end{proof}\begin{lem}
 \begin{align}\label{charac1}
\ch(A;\cone(M,\cone(C_1,\dots,C_m)^\dagger))&=\ch(A;\cone(M,C_1,\dots,C_m))\\ &= \ch(A;M)\cap_{i=1}^m \ch(A;C_i).\label{charac2}
\end{align}
\end{lem}
 \begin{proof}
Since $\cone(M,C_1,\dots,C_m)\subseteq \cone(M,\cone(C_1,\dots,C_m)^\dagger)$,  it is obvious that $\ch(A;\cone(M,\cone(C_1,\dots,C_m)^\dagger))\subseteq\ch(A;\cone(M,C_1,\dots,C_m)).$ 
Conversely, let $\chi \in\ch(A;\cone(M,C_1,\dots,C_m))$ and $c\in\cone(M,\cone(C_1,\dots,C_m)^\dagger).$ Then, $c=\sum_{i=1}^r a_ic_i$, with $a_i\in M$ and $c_i\in \cone(C_1,\dots,C_m)^\dagger$. For any $\varepsilon >0$, we have $a_i(c_i+\varepsilon) \in\cone(M,C_1,\dots,C_m)$, so 
\begin{align*}
\chi(c)-\varepsilon \sum_{i=1}^r \chi(a_i)=\sum_{i=1}^r \chi(a_i(c_i+\varepsilon)) \geq 0.
\end{align*}
 Letting $\varepsilon \to +0$ yields $\chi(c)\geq 0$, so  $\chi\in\ch(A;\cone(M,\cone(C_1,\dots,C_m)^\dagger)).$ This proves the  equality (\ref{charac1}).

A character $\chi$ of $A$ is in $\ch(A;\cone(M,C_1,\dots,C_m))$ if and only if $\chi(a c_1\cdots c_m)\newline =\chi(a)\chi(c_1)\cdots \chi(c_m)\geq 0$ for all $a\in M$, $c_1\in C_1,\dots, c_m\in C_m$, or equivalently, $\chi\in \ch(A;M)$ and $\chi \in \ch(A;C_i)$ for all $i=1,\dots,m$. This proves (\ref{charac2}).
\end{proof}

Now we are ready to prove the two theorems.\smallskip

{\bf Proof of Theorem \ref{daggerpre}:}

From Lemma \ref{product2} and \ref{squares}, $\cone(C_1,\dots,C_m)^\dagger$ is invariant under multiplication and contains all squares. Hence  it is a preordering. By Lemma \ref{archi}, it is Archimedean.

Since $\cone(C_1,\dots,C_m)^\dagger$ is invariant under multiplication, it is obvious that $\cone(M,\cone(C_1,\dots,C_m)^\dagger)$ is a $\cone(C_1,\dots,C_m)^\dagger$-module. \smallskip

{\bf Proof of Theorem \ref{positivst}:}

(i)$\to$(ii): Since $\cone(C_1,\dots,C_m)$ is  Archimedean  by Lemma \ref{archi}, so is the larger $\cone(M,\cone(C_1,\dots,C_m)^\dagger)$. Hence $ \ch(A;\cone(M,\cone(C_1,\dots,C_m)^\dagger))$ is compact in the weak topology.  Hence the positive continuous function $a\to \chi(a)$ on the compact space  $ \ch(A;\cone(M,\cone(C_1,\dots,C_m)^\dagger))=\ch(A;M)\cap_{i=1}^m \ch(A;C_i)$ (by (\ref{charac1}) and  (\ref{charac2}))  has a positive infimum, say $ \alpha$. Hence $a-\alpha/2$ satisfies condition (i) as well.  From Theorem \ref{daggerpre} we know that   $\cone(C_1,\dots,C_m)^\dagger$ is an Archimedean preordering and $\cone(M,\cone(C_1,\dots,C_m)^\dagger)$ is a 
$\cone(C_1,\dots,C_m)^\dagger$-module. Therefore, by the Positivstellensatz for modules of  Archimedean   semirings  (Proposition \ref{possemiring}), (i) implies that $a-\alpha/2$ belongs to 
$\cone(M,\cone(C_1,\dots,C_m)^\dagger)$,   that is, we have $a-\alpha/2  =\sum_{i=1}^s a_ic_i$, with $a_i\in M$ and $c_i\in   \cone(C_1,\dots,C_m)^\dagger$ for $i=1,\dots,s$. Using once more that $\cone(C_1,\dots,C_m)$ is Archimedean (Lemma \ref{archi})  there exists  $\lambda >0$ such that $(\lambda-\sum_{i=1}^s a_i) \in \cone(C_1,\dots,C_m)\subseteq \cone(M,C_1,\dots,C_m)$. Set   $\varepsilon:= \alpha/4$  and  $\delta:=\alpha(4 \lambda)^{-1}$. Then $\alpha/2-\delta\lambda=\varepsilon$. Since 
$c_i+\delta \in \cone(C_1,\dots,C_m)$, we therefore obtain 
\begin{align*}
a=\alpha/2-\delta\lambda + \sum_{i=1}^s a_i( c_i+\delta)  + \delta \Big(\lambda- \sum_{i=1}^s a_i\Big) \in \varepsilon +\cone(M,C_1,\dots,C_m).
\end{align*}

(ii)$\to$(i): Suppose that (ii) holds, say $a=\varepsilon +c$ with $c\in\cone(M,C_1,\dots,C_m) $. Then $\chi(a)=\varepsilon +\chi(c)\geq \varepsilon >0$ for all $\chi\in\ch(A;\cone(M,C_1,\dots,C_m))=\ch(A;M) \cap_{i=1}^m \ch(A;C_i)$ by (\ref{charac2}). This proves (i).

\section{Examples}\label{examples}

Throughout this section, we abbreviate $\dR_{n}[\underline{x}]:=\dR[x_1,x_2,\dots,x_n]$.  The characters of $A:=\dR_n[\underline{x}]$ are precisely the point evaluations at points of $\dR^n$.  For notational simplicity we identify a point evaluation with the corresponding point of $\dR^n$. Thus, if $C_i$ is a semiring or a quadratic module of $A_i$ generated by polynomials $f_1,\dots, f_r\in \dR_n[\underline{x}]$,  it is easily seen that  $\ch(A;C_i)$ is the {\it semi-algebraic set}
\begin{align}
K(f_1,\dots,f_r):= \{ t\in \dR^n: f_1(t)\geq 0,\dots,f_r(t)\geq 0\}.
\end{align}

In this section  we develop a number of  examples. Our particular aim  is to obtain representations of positive polynomials based on smaller sets of polynomials. For instance, instead of  squares of arbitrary elements of $A$ in many examples only squares of elements of certain subalgebras  occur. 

 All assertions  in the examples follow from Theorem \ref{positivst} or  from Corollary \ref{casem1} or   Proposition \ref{possemiring}. In most cases we leave the details of verifications to the reader.

We begin with some simple examples around cubes $[-1,1]^n$ for $n\in \dN$. 
\begin{exm} Multi-dimensional Bernstein theorem on $[-1,1]^n$\\
Recall that {\it Bernstein's theorem} says that each polynomial $q\in \dR[x]$ which is positive on the interval $[-1,1]$ is of the form
$$
\sum_{k,l=0}^r \alpha_{kl} (1-x)^k(1+x)^l, ~~\textrm{where}~~ \alpha_{kl}\geq 0,  k,l=1,\dots,r,~~\textrm {and} ~ r\in \dN_0
$$

Now let $n\in \dN$ and set $A=\dR_n[\underline{x}]$. If $p\in \dR_n[\underline{x}]$ and $p(x)>0$ for all $x\in [-1,1]^n$, then $p$ is a sum of terms
\begin{align}\label{bernstein}
\alpha (1-x_1)^{k_1}(1+x_1)^{l_1}\cdots(1-x_n)^{k_n}(1+x_n)^{l_n},
\end{align}
where $\alpha\geq 0$ and  $k_1,\dots,k_n,l_1,\dots,l_n\in \dN_0$.

This is an immediate consequence of  Theorem \ref{positivst}, applied with $A_i=\dR[x_i]$ and the semiring $C_i$ of $A_i$ generated by $1+x_i, 1-x_i$.  It follows also directly  the Archimedean Positivstellensatz for semirings (Proposition \ref{possemiring}), applied to the Archimedean semiring defined the sum of terms (\ref{bernstein}). 
\end{exm}
\begin{exm}\label{posrectangle1}  Multi-dimensional Markoff-Lukacs theorem on $[-1,1]^n$ \\
First let us state the one-dimensional  Markoff-Lukacs theorem. A polynomial $q\in \dR[x]$ is non-negative on $[-1,1]$ if and only if $q$ belongs to
\begin{align}\label{markoff1}
C:=\{ (1-x^2)q_1(x)^2+p(x)^2\, : q_1, p_1\in \dR[x]\}.
\end{align}
(It should be emphasized that we have single squares in (\ref{markoff1}) rather than sums of squares. Further, this result implies at once that $C$ is a preordering in $\dR[x]$.)

Now we turn an $n$-dimensional version. 
 If $p\in \dR_n[\underline{x}]$ satisfies $p(x)>0$ for all $x\in [-1,1]^n$, then $p$ is a sum  of terms 
\begin{align*}
\alpha  [(1-x_1^2)q_1(x_1)^2&+p_1(x_1)^2]\cdots [(1-x_n^2)q_n(x_n)^2+p_n(x_n)^2],
\end{align*}
where $\alpha\geq 0$  and $q_{1},p_{1},\dots,q_n,p_n\in \dR[x].$ 

This follows  from Theorem \ref{positivst}, applied with $A_j=\dR[x_j]$ and 
\begin{align}\label{markoff}
C_j:=\{ (1-x_{j}^2)q_{j}(x_{j})^2+p_{j}(x_j)^2\, : q_j,p_j\in \dR[x_j]\}.
\end{align}
\end{exm}

\begin{exm}\label{posrectangle2}  Combining Bernstein and Markoff-Lukacs theorems on $[-1,1]^{r+s}$ \\
Let $R$ be the cube $[-1,1]^{r+s}$ in $\dR^{s+k}$, where $r,s\in \dN$ . If $p\in \dR_{r+s}[\underline{x}]$ satisfies $p(x_1,\dots,x_{r+s})>0$ for all $(x_1,\dots,x_{r+s})\in R$, then $p$ is a sum  of terms 
\begin{align*}
\alpha (1&-x_1)^{k_1}(1+x_1)^{l_1}\cdots(1-x_r)^{k_r}(1+x_r)^{l_r}\times \\ [(1-x_{r+1}^2)q_{1}(x_{r+1})^2&+p_{1}(x_{r+1})^2]\cdots [(1-x_{r+s}^2)q_s(x_{r+s})^2+p_{s}(x_{r+s})^2],
\end{align*}
where $\alpha\geq 0$, $k_1,\dots,k_r,l_1,\dots,l_r\in \dN_0$ and $q_{1},p_{1},\dots,q_{s},p_{s}\in \dR[x].$ 

This follows  from Theorem \ref{positivst}, applied with $A_i=\dR[x_i]$. For $i=1,\dots,r,$ we let $C_i$
be  the Archimedean semiring of $A_i$ generated by  $1-x_i$,  $1+x_i$ and for $j=r+1,\dots,r+s$,  $C_j$ is the Archimedean preordering $C_j$ in $A_j$ defined by (\ref{markoff}).

This result can be also obtained  from Proposition \ref{possemiring}, applied to the Archimedean semiring of $\dR_{r+s}[\underline{x}]$ generated by the semirings $C_i$, $i=1,\dots,r+s$.
\end{exm}
In the following two examples we use the unit ball $$B^s:=\{(x_1,\dots,x_s)\in \dR^s:  x_1^2+\cdots+x_s^2\leq 1 \}.$$
\begin{exm}\label{recball}  $[-1,1]^r\times B^s$\\
Let $r,s\in \dN$ and  $A:= \dR_{r+s}[\underline{x}]$. We set $R=[-1,1]^r\times B_s$, that is,
\begin{align*}
R:=\{(x_1,\dots,x_{r+s})\in \dR^{r+s}: -1\leq x_1\leq 1,\dots, -1\leq x_r\leq 1, x_{r+1}^2+\cdots+x_{r+s}^2\leq 1 \}.
\end{align*}
 If $p(x)>0$ for all $x\in R$,  then $p$ is a nonnegative combination of terms 
\begin{align}
(1-x_1)^{k_1}(1+x_1)^{l_1}\cdots &(1-x_r)^{k_r}(1+x_r)^{l_r}\times\nonumber \\
[(1-x_{r+1}^2-\cdots-x_{r+s}^2)&q(x_{r+1},\dots,x_{r+s})+ p(x_{r+1},\dots,x_{r+s})],\label{bs1}
\end{align}
where $k_1,\dots,k_r,l_1,\dots,l_r\in \dN_0$ and $q,p\in \sum \dR[x_{r+1},\dots,x_{r+s}]^2.$

 \end{exm}

\begin{exm}\label{ball} $B^s$ revisited\\
The cone spanned by the  terms in (\ref{bs1}) is an Archimedean quadratic module of  $\dR[x_{r+1},\dots,x_{r+s}]$ which contains all strictly positive polynomials on the corresponding  unit ball.  In this example we recall   \cite[Example 4.1]{kss}  which gives an Archimedean semiring for unit ball in $B_s$. 

Let $S$ denote the semiring of $A=\dR[x_1,\dots,x_s]$ generating by the  polynomials
  \begin{align}
    f(x) := 1-x_1^2-\cdots-x_s^2,~
    g_{j,\pm}(x) := (1\pm x_j)^2,\;
    j=1,\dots,s.
  \end{align}
  Clearly,  $\ch(A;S)$ corresponds to the unit ball
  \begin{equation*}
B^s\equiv K(f)=\{x\in \dR^d: x_1^2+\cdots+x_s^2\leq 1\}.
  \end{equation*}
As noted in \cite{kss},  $S$ is Archimedean. Therefore, by Proposition \ref{possemiring} , if  $p\in A$ and  $p(x)>0$ for all $x\in B^s$, then $p$ is a non-negative combination of terms 
  \begin{align}\label{bs2}
 (1-x_1^2-\cdots-x_d^2)^n(1-x_1)^{2k_1}(1+x_1)^{2\ell_1}\cdots (1-x_d)^{2k_d}(1+x_d)^{2\ell_d},
 \end{align} 
 with\, $n, k_1,\ell_1,\dots,k_d,\ell_s\in \dN_0$. 

 Inserting  (\ref{bs2}) (with $x_1,\dots,x_s$ replaced by $x_{r+1},\dots,x_{r+s}$) instead of (\ref{bs1}) we obtain another description of strictly positive polynomials on $[-1,1]^r\times B^s$. 

For later use it is convenient to have this result also for the  ball
$$
B^s_\rho(x_0):=\{\{(x_1,\dots,x_s)\in \dR^s:  (x_1-x_{01})^2+\cdots+(x_s-x_{0s})^2\leq \rho^2 \}
$$
with center $x_0=(x_{01},\dots,x_{0s})\in \dR^s$ and radius $\rho>0$. This case is obtained from the unit ball $B^s$ be a linear transformation $x_i\to \rho^{-1}(x_i-x_{0i})$. Thus,  if  $p\in A$ satisfies  $p(x)>0$ for all $x\in B^s_\rho(x_0)$, then $p$ is a non-negative combination of polynomials
 \begin{align*}
\big (\rho^2-(x_1-x_{01})^2-\cdots-&(x_s-x_{0s})^2\big)^n \big(\rho-(x_1-x_{01})\big)^{2k_1}\big(\rho+(x_1-x_{01})\big)^{2\ell_1}\times \\ &\cdots \times \big(\rho-(x_s-x_{0s})\big)^{2k_s}\big(\rho+(x_s-x_{0s})\big)^{2\ell_s},
 \end{align*} with\, $n, k_1,\ell_1,\dots,k_d,\ell_s\in \dN_0$. 

  \end{exm}
\begin{exm} {\it Part of the interior of the elliptic paraboloid $x_3\geq x_1^2+x_2^2$, $x_3\leq 1$}\\
Let $C$ denote the semiring of $A:=\dR[x_1,x_2,x_3]$ generated by $x_3-x_1^2-x_2^2$ and $x_3$, $1-x_3$,  $(1-x_1)^2, (1+x_1)^2, (1-x_2)^2, (1+x_2)^2$. Clearly, $\ch(A;C)$ is the semialgebraic set
$$
R=\{(x_1,x_2,x_3)\in \dR^3: x_3\geq x_1^2+x_2^2,~  x_3\leq 1\, \}.
$$
For $i,j\in \{1,2\}$, $i\neq j$, we have
\begin{align}\label{idex3}
2+ x_3\pm 2 x_j=x_3-x_1^2-x_2^2 +\frac{1}{2}[(1+x_i)^2+(1-x_i)^2]+ (1\pm x_j)^2 \in C.
\end{align}
Since $x_3, 1-x_3 \in C$, $x_3\in A_{\rm bd}(C)$. Hence  $3\pm 2 x_j=(2+x_3 \pm 2 x_j)+ (1-x_3) \in C$ by  (\ref{idex3}), so that $x_j\in A_{\rm bd}(C)$ for $j=1,2$. Therefore, the semiring $C$ of $A$ is Archimedean.

Let $p\in \dR[x_1,x_2,x_3]$. If $q(x_1,x_2,x_3)>0$ for all $(x_1,x_2,x_3)\in R$, it follows from Proposition \ref{possemiring} that $q$ is a non-negative combination of polynomials
\begin{align*}
(x_3-x_1^2-x_2^2)^n (1-x_3)^{k_3} x_3^{l_3}  (1-x_1)^{2k_2} (1+x_1)^{2l_2} (1-x_2)^{2k_1} (1+x_2)^{2l_1},
\end{align*} 
where $n, k_1,l_1,k_2, l_2, k_3, l_3\in \dN_0$. 
\end{exm}
\begin{exm}\label{n3nosemiring}
In this example we set $A:=\dR[x_1,x_2,x_3]$. Further, let $A_j:=\dR[x_j],$  $j=1,2,3,$ and $A_{12}:= \dR[x_1,x_2]$. Clearly,
\begin{align}
C_{12}:=\big\{ (1-x_1^2)p_1+ (1-x_2^2) p_2 +p_3: \, p_1,p_2,p_3\in \sum \dR[x_1,x_2]^2\, \big\}
\end{align}
is an Archimedean quadratic module of $A_{12}$ and 
\begin{align}
C_j:=\big\{ (1-x_j^2)q_1+  q_2: \, q_1,q_2\in \sum \dR[x_j]^2\, \big\}
\end{align}
is an Archimedean preordering of $A_j$.

The subalgebras $A_{12,}, A_3$ and the cones $C_{12}, C_3$ satisfies the assumptions of Theorem \ref{positivst}. The semi-algebraic set $\ch(A;C_{12})\cap \ch(A;C_3)$ in $\dR^3$ is $[-1,1]^3$ and $\cone(C_{12},C_3)$ consists of all sums of terms
\begin{align}\label{termsc12}
[ (1-x_1^2)p_1(x_1,x_2)+ (1-x_2^2) p_2(x_1,x_2) +p_3(x_1,x_2)]\, [(1-x_3^2)q_1(x_3)+  q_2(x_3) ],
\end{align}
where $ p_1,p_2,p_3\in \sum \dR[x_1,x_2]^2$ and $ q_1,q_2\in \sum \dR[x_3]^2$. In this case,  Theorem \ref{positivst} says that if $p\in \dR[x_1,x_2,x_3]$ is positive on $[-1,1]^3$, then $q$ is  a sum of terms (\ref{termsc12}). Note that $C_{12}$ is not a preordering and  $\cone(C_{12},C_3)$ is neither a quadratic module nor a semiring.

The subalgebras $A_1, A_2, A_3$ with cones $C_1, C_2, C_3$ fulfill the assumptions of Theorem \ref{positivst} as well. This is just the case $n=3$ of Example \ref{posrectangle1} and $\cone(C_1,C_2,C_3)$ is a semiring.
\end{exm}
\begin{exm}
We reconsider  \cite[Example 7.1]{kss} in the context of Theorem \ref{positivst}. Let $A=\dR[x_1,x_2]$ and set 
 \begin{equation*}
    f_1(x_1,x_2):= x_2-(x_1-1/2)^2,~~
    f_2(x_1,x_2):= x_2-x_1^2.
  \end{equation*}
For $i=1,2$, let $A_i$ be the subalgebra of $A$ generated by $f_i$ and $C_i$ the semiring of $A_i$ generated by $1-f_i, 1+f_i$. It is clear that $C_i$ is Archimedean in $A_i$ and that $A$ is generated by $A_1$ and $A_2$, so  the assumptions of Theorem \ref{positivst} are fulfilled. Then $\ch(A,\cone(C_1,C_2))$ corresponds to the  semialgebraic set  
  \begin{equation*}
    K:=\{ (x_1,x_2)\in \dR^2: {-1}\leq f_1(x_1,x_2)\leq 1,~ {-1}\leq f_2(x_1,x_2)\leq 1 \}.
  \end{equation*}
Therefore, by Theorem \ref{positivst},  
  each polynomial $p\in \dR[x_1,x_2]$ which is positive in all points of $K$ is a finite sum of terms
  \begin{equation*}
     \alpha (1-f_1(x_1,x_2))^i (1+f_1(x_1,x_2))^j (1-f_2(x_1,x_2)\big)^k (1+f_2(x_1,x_2))^\ell,
  \end{equation*}
  where  $\alpha\geq 0$ and $i,j,k,\ell\in \dN_0$. 
\end{exm}

Now suppose that $f_1,\dots,f_r$ is a finite set polynomials from $A:=\dR[x_1,\dots,x_n]$ such that the semi-algebraic set $K:=K(f_1,\dots,f_r)$ of $\dR^n$ is {\it compact}. 

It is well-known that the compactness of $K(f_1,\dots,f_r)$  does not imply that the  quadratic module $Q$ generated by $f_1,\dots,f_r$ is Archimedean, so  in general the Archimedean Positivstellensatz for quadratic modules does not apply. Similarly, if the semi-algebraic set obtained from a semiring $S$ is compact, then $S$ is not necessarily Archimedean.  In order to apply the corresponding Archimedean Positivstellensatz one has to enlarge $Q$ or $S$, respectively.
As noted above, we want to  obtain representations  for the positive polynomials on $K(f_1,\dots,f_r)$ with "small" classes of squares or polynomials. To achieve this goal this we  propose to proceed as follows. 

We choose subalgebras $A_i$, $i=1,\dots,m$, and Archimedean quadratic modules or semirings $C_i$ of $A_i$ such that $A$ is generated by $A_1,\dots,A_m$ and the semi-algebraic set corresponding to $\cap_{i=1}^m \ch(A;C_i)$ contains $K(f_1,\dots,f_r)$. This is possible, because $K(f_1,\dots,f_r)$ is compact. Let $M$ denote  the unital cone of $A$ defined by
\begin{align}
M:=\{ \lambda_0 f_0+\lambda_1 f_1+\cdots+ \lambda_r f_r:\, \lambda_0\geq 0, \lambda_1\geq 0,\dots,\lambda_r\}, ~~\textrm{where}~ f_0:=1.
\end{align}
Then 
\begin{align}\label{characterK}
\ch(A;M)\cap_{i=1}^m \ch(A;C_i)= \ch(A;M)=K(f_1,\dots,f_r).
\end{align}
Therefore, by Theorem \ref{positivst}, if  $p\in \dR_n[\underline{x}]$ and $p(x)>0$ for all $x\in K(f_1,\dots,f_r)$, then $p\in \cone(M,\cone(C_1,\dots,C_m))$, that is, $p$ is a sum  of polynomials
\begin{align}\label{conem}
 f_j\, \alpha\, c_1\cdots c_m, ~~~ {\textrm where}~~~ \alpha\geq 0,~ c_1\in C_1,\dots,c_m\in C_m~~ \textrm{and}~ ~j=0,\dots ,r.
\end{align}
Note that in (\ref{conem}) there are no mixed products of the polynomials $f_1,\dots,f_r$.

That is, the algebras  $A_i$ and  cones $C_i$ are choosen such that $\cone(C_1,\dots,C_m)$ is {\it Archimedean} and $K(f_1,\dots,f_r)\subseteq\cap_{i=1}^m \ch(A;C_i)$.  The terms $f_j$ in  (\ref{conem}) are only needed to ensure  that the  semi-algebraic set $K(f_1,\dots,f_r)$ is the corresponding set (\ref{characterK})  of characters.

There is a lot of freedom in choosing the cones $C_i$. For instance, one may have
 an $n$-dimensional cube,  rectangle, or  ball as corresponding set $\cap_{i=1}^m \ch(A;C_i)$  and then  apply Examples
 \ref{posrectangle1}, \ref{posrectangle2}, \ref{recball}, or \ref{ball}. We illustrate this by a very simple example, the famous Jacobi-Prestel counter-example \cite[Example 4.6]{jp}. 
\begin{exm} $n=2$, $f_1= x_1-\frac{1}{2}, f_2=x_2-\frac{1}{2}, f_3=1-x_1x_2$.\\
Then $K:=K(f_1,f_2,f_3)$ is a compact set bounded by the hyperbola $1=x_1x_2$ and the two lines $x_1=\frac{1}{2}$,  $x_1=\frac{1}{2}$.  As noted in \cite{jp}, 
the quadratic module of $\dR[x_1,x_2]$ generated by $f_1,f_2,f_3$ is not Archimedean . How to repair this?

 We set $A_1=\dR[x_1], A_2=\dR[x_2]$. and consider two embeddings of  $K(f_1,f_2,f_3)$.

Case 1:  Embed  $K(f_1,f_2,f_3)$ into the square $[1/2,2]\times [1/2,2].$  

Then there are two possibilities. For the first  we take $C_i$, $i=1,2,$ to be  the semiring generated by $2-x_i$ and $x_i-\frac{1}{2}$. Clearly, $C_i$ is an Archimedean semring of $A_i$. Then $\cone(M,\cone(C_1,C_2))$ is the unital cone spanned by $2-x_1, 2-x_2, x_1-\frac{1}{2}, x_2-\frac{1}{2}$ and $f_3$. Therefore, if $p\in\dR[x_1,x_2]$ and $p(x)>0$ for all $x\in K(f_1,f_2,f_3)$, then $p$ is a non-negative combination of terms
\begin{align}
&(2-x_1)^{k_1} ( x_1-1/2)^{l_1} (2-x_2)^{k_2} ( x_2-1/2)^{l_2}~\textrm{and}\\  &f_3 \,(2-x_1)^{k_1} ( x_1-1/2)^{l_1} (2-x_2)^{k_2} ( x_2-1/2)^{l_2}, 
\end{align}
where $k_1,k_2,l_1,l_2\in \dN_0$.

For the second possibility we let $C_i$, $i=1,2$, be the Archimedean preordering 
\begin{align}\label{markoff}
C_i:=\{ (2-x_i)(x_i-1/2)q_{i}(x_i)^2+p_i(x_i)^2\, : q_i,p_i\in \dR[x_i]\}.
\end{align}
in $A_i$. Then,  by Theorem \ref{positivst}, if  $p(x)>0$ for all  $x\in K(f_1,f_2,f_3)$, then $p$ is a non-negative combination of polynomials
\begin{align*}
f_i (x_1,x_2) [(2{-}x_1)(x_1{-}1/2)q_1(x_1)^2+p_1(x_1)^2]\, ((2{-}x_2)(x_2{-}1/2)q_2(x_2)^2+p_2(x_2)^2],
\end{align*}
where $i=0,1,2,3$ and $q_1,p_1,q_2,p_2\in \dR[x]$. Here we have set $f_0=1$.
\smallskip

Case 2: Embed $ K(f_1,f_2,f_3)$ into the dics  $B^2_\rho(x_0)$ with center  $x_0=(5/4,5/4)$ and radius  $\rho=\sqrt{2}\, \frac{ 3}{4}=\sqrt{18}/4$. 

Then,  if  $p\in  \dR[x_1,x_2]$ satisfies  $p(x_1,x_2)>0$ for all $(x_1,x_2)\in K(f_1,f_2,f_3)$,  it follows from (\ref{conem}) and  the formula for  $B^2_\rho(x_0)$ in   Example \ref{ball} that $p$ is a sum of polynomials
 \begin{align*}
f_j\, \alpha \big(9/8-(x_1{-}5/4)^2&-(x_2{-}5/4)^2\big)^n \big(\sqrt{18} /4-(x_1{-}5/4)\big)^{2k_1}\big(\sqrt{18}/4+(x_1{-}5/4)\big)^{2\ell_1}\\ & \times 
\big(\sqrt{18}/4-(x_2{-}5/4)\big)^{2k_2}\big(\sqrt{18}/4+(x_2{-}5/4)\big)^{2\ell_2},
 \end{align*} 
where\, $\alpha\geq 0$, $j=1,2,3$ and $n, k_1,\ell_1,k_2,\ell_2\in \dN_0$. 
\end{exm}
\begin{rem}
In this paper we have tried to express positive polynomials on compact semi-algebraic sets such that small classes of elements or squares are needed.  There are  a number of natural questions on this matter. 
For instance, one may ask whether the following is true: Each polynomial which is positive on the square $[0,1]^2$  is a finite sum of terms
\begin{align}
(1-x_1^2)p_1(x_1)^2q_1(x_2)^2+(1-x_2^2)p_2(x_1)^2q_2(x_2)^2+p_3(x_1)^2q_3(x_2)^2,
\end{align}
where $p_1,p_2,p_3,q_1,q_2,q_3\in \dR[x]$.
\end{rem}

\bibliographystyle{amsalpha}

\end{document}